\documentclass{article}%
\usepackage{amsmath}
\usepackage{amsfonts}
\usepackage{amssymb}
\usepackage{graphicx}%
\providecommand{\U}[1]{\protect\rule{.1in}{.1in}}
\newtheorem{theorem}{Theorem}

\newenvironment{proof}[1][Proof]{\noindent\textbf{#1.} }{\ \rule{0.5em}{0.5em}}
\begin{document}

\title{\emph{2}-Degenerate Bertrand curves in Minkowski spacetime}
\author{Mehmet G\"{o}\c{c}men and Sad\i k Kele\c{s}\\Department of Mathematics, Faculty of Arts and Sciences, \\\.{I}n\"{o}n\"{u} \ University 44280 Malatya Turkey.}
\maketitle

\begin{abstract}
In this paper we define a new type of \emph{2}-degenerate Cartan curves in
Minkowski spacetime $\left(  R_{1}^{4}\right)  $. We prove that this type of
curves contain only the polynomial functions as its components whose third
derivative vanish completely. No curve with acceleration zero in $R_{1}^{4}$
is a \emph{2}-degenerate Cartan curve, therefore we show that the type of
curves that we search for must contain polynomials of degree two among its components.

\end{abstract}

\begin{center}

\end{center}

\section{Introduction}

\noindent\ \ \ \ H. Matsuda and S. Yorozu \cite{matsuda} introduced a new type
of curves called special Frenet curves and proved that a special Frenet curve
in $R^{n}$ is not a Bertrand curve if $n\geqslant4.$ They also improved an
idea of generalized Bertrand curve in $R^{4}$. A. Ferrandez, A. Gimenez, P.
Lucas \cite{s-degenerate} introduced the notion of \emph{s}-degenerate curves
in Lorentzian space forms. They obtained a reference along an\emph{\ s}%
-degenerate curve in an n-dimensional Lorentzian space with the minimum number
of curvatures. That reference generalizes the reference of Bonnor for null
curves in Minkowski spacetime and it would be called the Cartan frame of the
curve. The associated curvature functions are called the Cartan curvatures of
the curve. They characterized the \emph{s}-degenerate helices ( i.e,\emph{\ s}%
-degenerate curves with constant Cartan curvatures ) in n-dimensional
Lorentzian space forms and they obtained a complete classification of them in
dimension four.

Let $C$ be an \emph{s}-degenerate Cartan curve in $R_{1}^{n}.$ We call
$\ W_{j}$ the spacelike Cartan j-normal vector along $C$, and the spacelike
Cartan j-normal line of $C$ at $c(s)$ is a line generated by $W_{j}(s)$
through $c(s)$ $\left(  j=1,2,...,n-2\right)  .$ The spacelike Cartan $\left(
j,k\right)  $-normal plane of $C$ at $c(s)$ is a plane spanned by $W_{j}(s)$
and $W_{k}(s)$ through $c(s)$ $\left(  j,k=1,2,...,n-2;j\neq k\right)  .$ In
this paper we characterize \emph{2}-degenerate (1,2)-Bertrand curves in
Minkowski spacetime.

\section{Preliminaries}

\ \ \ \ Let E be a real vector space with a symmetric bilinear mapping
$g:E\times E\rightarrow R.$ We say that $g$ is degenerate on $E$ if there
exist a vector $\varepsilon\neq0$ in $E$ such that%

\[
g(\varepsilon,\nu)=0\ \ \ \text{for all }\nu\in E,
\]

\noindent otherwise, $g$ is said to be non-degenerate. The radical (also
called the null space ) of $E$, with respect to $g$, is the subspace Rad($E$)
of $E$ defined by%

\[
\text{Rad(}E\text{)}=\left\{  \varepsilon\in E\text{ such that }%
g(\varepsilon,\nu)=0,\text{ \ }\nu\in E\right\}  .
\]

\noindent For simplicity, we will use $\left\langle ,\right\rangle $ instead
of $g$. A vector $\nu$ is said to be timelike, lightlike or spacelike provided
that $g\left(  \nu,\nu\right)  <0,$ $g\left(  \nu,\nu\right)  =0$ ( and
$\nu\neq0$), or $g\left(  \nu,\nu\right)  >0$ respectively. The vector $\nu=0$
is said to be spacelike. A unit vector is a vector $u$ such that $g\left(
u,u\right)  =\mp1.$ Two vectors $u$ and $v$ are said to be orthogonal, written
$u\bot v,$ if $g\left(  u,v\right)  =0$.

Let $\left(  M_{1}^{n},\nabla\right)  $ be an oriented Lorentzian manifold and
let $C:I\rightarrow M_{1}^{n}$ be a differentiable curve in $M_{1}^{n}.$ For
any vector field $V$ along $C$, Let $V^{\prime}$ be the covariant derivative
of $V$ along $C.$ Write%
\[
E_{i}\left(  t\right)  =span\left\{  c^{\prime}\left(  t\right)
,c^{\prime\prime}\left(  t\right)  ,...,c^{\left(  i\right)  }\left(
t\right)  \right\}  ,
\]

\noindent where $t\in I$ and $i=1,2,...,n.$ Let $d$ be the number defined by%
\[
d=\max\left\{  i:\dim E_{i}\left(  t\right)  =i\text{ \ for all }t\right\}  .
\]

With the above notation, the curve $C:I\rightarrow M_{1}^{n}$ \ is said to be
an \emph{s}-degenerate\textit{\ }( or\emph{\ s}-lightlike) curve if for all
$1\leqslant i\leqslant d$, $\dim Rad\left(  E_{i}\left(  t\right)  \right)  $
is constant for all $t$, and there exist $s$, $0\leqslant s\leqslant d$, such
that $Rad\left(  E_{s}\right)  \neq\left\{  0\right\}  $ and $Rad\left(
E_{j}\right)  =\left\{  0\right\}  $ for all \ $j<s.$ Note that \emph{1}%
-degenerate curves are precisely the null (or lightlike) curves. In this paper
we will focus on \emph{2}-degenerate curves $(s=2)$ in Minkowski spacetime.
Notice that they must be spacelike curves

A \emph{spacetime} is a connected time-oriented four dimensional Lorentz
manifold. A \emph{Minkowski spacetime }$M$ is a spacetime that is isometric to
Minkowski 4-space $R_{1}^{4}$ $\cite{o'neil}$. So $R_{1}^{4}$ is a
4-dimensional Lorentz manifold furnished with the metric $\left\langle
,\right\rangle $ defined as follows%
\[
\left\langle x,y\right\rangle =-x^{0}y^{0}+x^{1}x^{1}+x^{2}x^{2}+x^{3}x^{3}%
\]

\noindent for all vectors $x,y\in R_{1}^{4}$; $x=\left(  x^{0},x^{1}%
,x^{2},x^{3}\right)  ,$ $y=\left(  y^{0},y^{1},y^{2},y^{3}\right)  ,$
$x^{i},y^{i}\in R$, $0\leqslant i\leqslant3.$

Let $C$ be a 2-degenerate Cartan curve in $R_{1}^{4}$. Then the Cartan
equations are in the following form \cite{s-degenerate}.%

\begin{align*}
c^{\prime}  &  =W_{1},\\
W_{1}^{\prime}  &  =L,\\
L^{\prime}  &  =k_{1}W_{2},\\
W_{2}^{\prime}  &  =-k_{2}L+k_{1}N,\\
N^{\prime}  &  =W_{1}-k_{2}W_{2}.
\end{align*}

\noindent where $L,N$ are null, $\left\langle L,N\right\rangle =-1,$ $\left\{
L,N\right\}  $ and $\left\{  W_{1},W_{2}\right\}  $ are orthogonal,
\newline$\left\{  W_{1},W_{2}\right\}  $ is orthonormal. $\left\{
L,N,W_{1},W_{2}\right\}  $ is positively oriented. We assume the set $\left\{
c^{\prime},c^{\prime\prime},c^{\prime\prime\prime},c^{\left(  4\right)
}\right\}  $ has the same orientation with the set $\left\{  L,N,W_{1}%
,W_{2}\right\}  ,$ so we get $k_{1}<0.$

Let $\left(  C,\overline{C}\right)  $ be a pair of \ framed null Cartan curves
in $R_{1}^{4}$, with pseudo-arc parameters $s$ and $\overline{s},$
respectively. This pair is said to be a null Bertrand pair if their spacelike
vectors $W_{1}$ and $\overline{W_{1}}$ are linearly dependent. The curve
$\overline{C}$ is called a Bertrand mate of $C$ and vice versa. A framed null
curve is said to be a null Bertrand curve if it admits a Bertrand mate
\cite{duggal}. To be precise, a null Cartan curve $C$ in $R_{1}^{4}$ $\left(
c:I\rightarrow R_{1}^{4}\right)  $ is called a Bertrand curve if there exist a
null Cartan curve $\overline{C}$ $\left(  \overline{c}:\overline{I}\rightarrow
R_{1}^{4}\right)  ,$ distinct from $C,$ and a regular map $\varphi
:I\rightarrow\overline{I}\left(  \overline{s}=\varphi\left(  s\right)
,\frac{d\varphi\left(  s\right)  }{ds}\neq0\text{ for all }s\in I\right)  $
such that the spacelike vectors $W_{1}$ of $C$ and $\overline{W_{1}}$ of
$\overline{C}$ are linearly dependent at each pair of corresponding points
$c\left(  s\right)  $ and $\overline{c}\left(  \overline{s}\right)
=\overline{c}\left(  \varphi\left(  s\right)  \right)  $ under $\varphi.$

\section{(1,2)-Bertrand curves in $R_{1}^{4}$}

\ \ \ \ Let $C$ and $\overline{C\text{ }}$be 2-degenerate Cartan curves in
$R_{1}^{4}$ and $\varphi:I\rightarrow\overline{I}$ a regular map $\left(
\overline{s}=\varphi\left(  s\right)  ,\frac{d\varphi\left(  s\right)  }%
{ds}\neq0\text{ for all }s\in I\right)  $ such that each point $c\left(
s\right)  $ of $C$ corresponds to the point $\overline{c}\left(  \overline
{s}\right)  $ of $\overline{C}$ under $\varphi$ for all $s\in I.$ Here $s$ and
$\overline{s}$ are pseudo-arc parameters of $C$ and $\overline{C}$
respectively. If the Cartan (1,2)-normal plane at each point $c\left(
s\right)  $ of $C$ coincides with the Cartan (1,2)-normal plane at
corresponding point $\overline{c}\left(  \overline{s}\right)  =\overline
{c}\left(  \varphi\left(  s\right)  \right)  $ of $\overline{C}$ for all $s\in
I,$ then $C$ is called the (1,2)-Bertrand curve in $R_{1}^{4}$ and
$\overline{C}$ is called the (1,2)-Bertrand mate of $C.$

\begin{theorem}
Let $C$ be a \emph{2}-degenerate Cartan curve in $R_{1}^{4}$ with curvature
functions $k_{1},k_{2}$. Then $C$ is a (1,2)-Bertrand curve if and only if
there are polynomial functions $\alpha$ and $\beta$ satisfying
\begin{align}
\beta\left(  s\right)   &  \neq0\tag*{(a)}\\
k_{1}\left(  s\right)   &  =0\tag*{(b)}\\
k_{2}\left(  s\right)   &  =\frac{\alpha\left(  s\right)  }{\beta\left(
s\right)  }\tag*{(c)}\\
\beta^{\prime}\left(  s\right)   &  \neq0\tag*{(d)}\\
\left(  1+\alpha^{\prime}\left(  s\right)  \right)  ^{2}+\left(  \beta
^{\prime}\left(  s\right)  \right)  ^{2}  &  \neq0\tag*{(e)}\\
\max\deg\left\{  \alpha\left(  s\right)  \right\}   &  =1\tag*{(f)}\\
\deg\left\{  \beta\left(  s\right)  \right\}   &  =1 \tag*{(g)}%
\end{align}
for all $s\in I.$ By \emph{(f)}, we mean the maximum degree of the set
containing the polynomial function $\alpha$ is one.
\end{theorem}

\begin{proof}
$\Rightarrow)$: Assume that $C$ is a (1,2)-Bertrand curve, then we can write

\begin{equation}
\overline{c}\left(  \overline{s}\right)  =\overline{c}\left(  \varphi\left(
s\right)  \right)  =c\left(  s\right)  +\alpha\left(  s\right)  W_{1}\left(
s\right)  +\beta\left(  s\right)  W_{2}\left(  s\right)  \label{s1}%
\end{equation}

\noindent And since the planes spanned by $\left\{  W_{1},W_{2}\right\}  $ and
$\left\{  \overline{W_{1}},\overline{W_{2}}\right\}  $ coincide, we can also
write%
\begin{align}
\overline{W_{1}}\left(  \overline{s}\right)   &  =\cos\theta\left(  s\right)
W_{1}\left(  s\right)  +\sin\theta\left(  s\right)  W_{2}\left(  s\right)
\label{s2}\\
\overline{W_{2}}\left(  \overline{s}\right)   &  =-\sin\theta\left(  s\right)
W_{1}\left(  s\right)  +\cos\theta\left(  s\right)  W_{2}\left(  s\right)  .
\label{h2}%
\end{align}

\noindent Notice that $\sin\theta\left(  s\right)  \neq0$ for all $s$ $\in I$.
Because if $\sin\theta\left(  s\right)  =0$, then we get the position
$\overline{W_{1}}\left(  \overline{s}\right)  =\mp W_{1}\left(  s\right)  .$
This implies that $C$ and $\overline{C}$ coincides. But we know that the
Bertrand mate $\overline{C}$ of $C$ must be distinct from $C$. So, we have
$\sin\theta\left(  s\right)  \neq0$ for all $s$ $\in I$. Now differentiating
(\ref{s1}) with respect to $s,$ we get%
\begin{align}
\overline{W_{1}}\left(  \overline{s}\right)  \frac{d\overline{s}}{ds}  &
=\left(  1+\alpha^{\prime}\left(  s\right)  \right)  W_{1}\left(  s\right)
+\left(  \alpha\left(  s\right)  -\beta\left(  s\right)  k_{2}\left(
s\right)  \right)  L\left(  s\right) \label{s3}\\
&  +\beta^{\prime}\left(  s\right)  W_{2}\left(  s\right)  +\beta\left(
s\right)  k_{1}\left(  s\right)  N\left(  s\right)  .\nonumber
\end{align}

\noindent Since we assume that the curve $C$ is (1,2)-Bertrand curve, then the
map $\varphi$ between the pseudo-arc parameters of $C$ and its Bertrand mate
$\overline{C}$ must be regular. So the following equation holds.%
\[
\frac{d\left(  \varphi\left(  s\right)  \right)  }{ds}=\frac{d\overline{s}%
}{ds}\neq0.
\]

\noindent From (\ref{s3}), we obtain\emph{\ }the relations (a), (b) and (c).
By the following facts%
\begin{align}
\frac{d\overline{s}}{ds}  &  =\left\langle \overline{W_{1}}\left(
\overline{s}\right)  ,\overline{W_{1}}\left(  \overline{s}\right)
\frac{d\overline{s}}{ds}\right\rangle =\left(  1+\alpha^{\prime}\left(
s\right)  \right)  \cos\theta\left(  s\right)  +\beta^{\prime}\left(
s\right)  \sin\theta\left(  s\right) \label{p1}\\
0  &  =\left\langle \overline{W_{2}}\left(  \overline{s}\right)
,\overline{W_{1}}\left(  \overline{s}\right)  \frac{d\overline{s}}%
{ds}\right\rangle =-\left(  1+\alpha^{\prime}\left(  s\right)  \right)
\sin\theta\left(  s\right)  +\beta^{\prime}\left(  s\right)  \cos\theta\left(
s\right)  \label{p2}%
\end{align}

\noindent we obtain%

\begin{align}
1+\alpha^{\prime}\left(  s\right)   &  =\frac{d\overline{s}}{ds}\cos
\theta\left(  s\right)  ,\label{s6}\\
\beta^{\prime}\left(  s\right)   &  =\frac{d\overline{s}}{ds}\sin\theta\left(
s\right)  . \label{f1}%
\end{align}

\noindent Since $\frac{d\overline{s}}{ds}\neq0$ and $\sin\theta\left(
s\right)  \neq0$ for all $s\in I$ in (\ref{f1}), we obtain the relation (d).
By using (\ref{s6}) and (\ref{f1}), we get%

\begin{equation}
\left(  1+\alpha^{\prime}\left(  s\right)  \right)  ^{2}+\left(  \beta
^{\prime}\left(  s\right)  \right)  ^{2}=\left(  \frac{d\overline{s}}%
{ds}\right)  ^{2}. \label{s7}%
\end{equation}

\noindent Using (\ref{s7}), we obtain (e). The Bertrand mate $\overline{C}$ of
$C$ is itself a Bertrand curve, therefore the curvature $\overline{k_{1}}$ of
$\overline{C}$ is also zero. This means that the components of the curve
$\overline{C}$ consists of polnomials whose third derivative with respect to
its pseudo-arc parameter $\overline{s}$ vanish completely. By\ using this
information and (\ref{s1}), it is obvious that the map $\varphi$ between the
pseudo-arc parameters of $C$ and $\overline{C}$ at corresponding points
$c\left(  c\right)  $ and $\overline{c}\left(  \overline{s}\right)  $
respectively, must be linear, that is, $\frac{d\varphi\left(  s\right)  }%
{ds}=\frac{d\overline{s}}{ds}$ must be a nonzero constant. Using this fact,
(\ref{s7}), and (d) we get the relations (f) and (g).

$\Leftarrow)$: Now let us think the contrary.

\noindent Let $C$ be a \emph{2}-degenerate Cartan curve in $R_{1}^{4}$ with
curvature functions $k_{1}$ and $k_{2}$ and assume that the relations
(a),(b),(c),(d),(e),(f) are satisfied for this curve.

Now define a curve $\overline{C}$ by%

\begin{equation}
\overline{c}\left(  s\right)  =c\left(  s\right)  +\alpha\left(  s\right)
W_{1}\left(  s\right)  +\beta\left(  s\right)  W_{2}\left(  s\right)
\label{s8}%
\end{equation}

\noindent where $s$ is the pseudo-arc parameter of $C$. Differentiating
(\ref{s8}) with respect to $s$, using the Frenet equations and the hypothesis
in the above, we obtain%

\begin{equation}
\frac{d\overline{c}\left(  s\right)  }{ds}=(1+\alpha^{\prime}\left(  s\right)
)W_{1}\left(  s\right)  +\beta^{\prime}\left(  s\right)  W_{2}\left(
s\right)  . \label{s9}%
\end{equation}

\noindent By using (\ref{s9}), we get%

\begin{equation}
\left(  \frac{d\overline{c}\left(  s\right)  }{ds}\right)  ^{2}=(1+\alpha
^{\prime}\left(  s\right)  )^{2}+\left(  \beta^{\prime}\left(  s\right)
\right)  ^{2}\neq0. \label{r1}%
\end{equation}

\noindent It is obvious from (\ref{r1}) that $\overline{C}$ is a regular
curve. Let us define a regular map $\varphi:s\rightarrow\overline{s}$ by%

\[
\overline{s}=\varphi\left(  s\right)  =%
{\displaystyle\int\limits_{0}^{s}}
\left\langle \frac{d\overline{c}\left(  s\right)  }{ds},\frac{d\overline
{c}\left(  s\right)  }{ds}\right\rangle ^{\frac{1}{2}}ds.
\]

\noindent where $\overline{s}$ denotes the pseudo-arc parameter of
$\overline{C}$. Then we obtain%

\begin{equation}
\frac{d\overline{s}}{ds}=\frac{d\varphi\left(  s\right)  }{ds}=\sqrt
{(1+\alpha^{\prime}\left(  s\right)  )^{2}+(\beta^{\prime}\left(  s\right)
)^{2}}>0. \label{s10}%
\end{equation}

\noindent Here, notice that%

\begin{equation}
\frac{d\overline{s}}{ds}=\lambda\label{k8}%
\end{equation}

\noindent is a nonzero constant.

\noindent Thus the curve $\overline{C}$ is rewritten as%

\begin{equation}
\overline{c}\left(  \overline{s}\right)  =\overline{c}\left(  \varphi\left(
s\right)  \right)  =c\left(  s\right)  +\alpha\left(  s\right)  W_{1}\left(
s\right)  +\beta\left(  s\right)  W_{2}\left(  s\right)  . \label{s11}%
\end{equation}

\noindent If we differentiate (\ref{s11}) with respect to $s$, use the Cartan
equations for the \emph{2}-degenerate curves in $R_{1}^{4}$ and the
hypothesis, we get%

\begin{equation}
\lambda\overline{W_{1}}\left(  \overline{s}\right)  =(1+\alpha^{\prime}\left(
s\right)  )W_{1}\left(  s\right)  +\beta^{\prime}\left(  s\right)
W_{2}\left(  s\right)  . \label{s12}%
\end{equation}

\noindent By using (\ref{s10}), (\ref{k8}) and (\ref{s12}), we can set%

\begin{equation}
\overline{W_{1}}\left(  \overline{s}\right)  =\cos\tau\left(  s\right)
W_{1}\left(  s\right)  +\sin\tau(s)W_{2}\left(  s\right)  \label{s13}%
\end{equation}

\noindent where%

\begin{align}
\cos\tau\left(  s\right)   &  =\frac{1+\alpha^{\prime}\left(  s\right)
}{\lambda},\label{k1}\\
\sin\tau(s)  &  =\frac{\beta^{\prime}\left(  s\right)  }{\lambda}. \label{k2}%
\end{align}

\noindent After differentiating (\ref{s13}) with respect to $s$, we get%
\begin{align}
\lambda\overline{L}\left(  \overline{s}\right)   &  =\frac{d\cos\tau\left(
s\right)  }{ds}W_{1}\left(  s\right)  +\frac{d\sin\tau\left(  s\right)  }%
{ds}W_{2}\left(  s\right) \label{s14}\\
&  +\left(  \cos\tau\left(  s\right)  -k_{2}\sin\tau\left(  s\right)  \right)
L\left(  s\right)  .\nonumber
\end{align}

\noindent Applying the metric $\left\langle ,\right\rangle $ on each side of
the equation (\ref{s14}), we get
\begin{equation}
(\cos^{\prime}\tau\left(  s\right)  )^{2}+(\sin^{\prime}\tau\left(  s\right)
)^{2}=0. \label{s15}%
\end{equation}

\noindent From (\ref{s15}), we get%

\[
\frac{d\cos\tau\left(  s\right)  }{ds}=\frac{d\sin\tau\left(  s\right)  }%
{ds}=0.
\]

\noindent So the $\tau\left(  s\right)  $ must be the constant function
$\tau_{0}.$ Thus we obtain%

\begin{align}
\cos\tau_{0}  &  =\frac{1+\alpha^{\prime}\left(  s\right)  }{\lambda
},\label{k3}\\
\sin\tau_{0}  &  =\frac{\beta^{\prime}\left(  s\right)  }{\lambda}. \label{k4}%
\end{align}

\noindent From (\ref{s13}), it holds%

\begin{equation}
\overline{W_{1}}\left(  \overline{s}\right)  =\cos\tau_{0}W_{1}\left(
s\right)  +\sin\tau_{0}W_{2}\left(  s\right)  . \label{k5}%
\end{equation}

\noindent Now the equation (\ref{s14}) becomes%

\begin{equation}
\lambda\overline{L}\left(  \overline{s}\right)  =\left(  \cos\tau_{0}%
-k_{2}\sin\tau_{0}\right)  L\left(  s\right)  . \label{s16}%
\end{equation}

\noindent Note that, since $\lambda\neq0$ in (\ref{s16}), we have%

\[
\cos\tau_{0}-k_{2}\sin\tau_{0}\neq0.
\]

\noindent So we can write the following%

\begin{equation}
\overline{N}\left(  \overline{s}\right)  \left(  \cos\tau_{0}-k_{2}\sin
\tau_{0}\right)  =\lambda N\left(  s\right)  . \label{s17}%
\end{equation}

\noindent If we differentiate (\ref{s16}) with respect to $s$, we get%
\begin{align}
\lambda^{2}\overline{k_{1}}\left(  \overline{s}\right)  \overline{W_{2}%
}\left(  \overline{s}\right)   &  =\frac{d\left(  \cos\tau_{0}-k_{2}\sin
\tau_{0}\right)  }{ds}L\left(  s\right) \label{s18}\\
&  +\left(  \cos\tau_{0}-k_{2}\sin\tau_{0}\right)  k_{1}\left(  s\right)
W_{2}\left(  s\right)  .\nonumber
\end{align}

\noindent If we use (b) ( $k_{1}\left(  s\right)  =0$ ) in (\ref{s18}), we
get
\begin{equation}
\overline{k_{1}}\left(  \overline{s}\right)  =0. \label{t1}%
\end{equation}

\noindent And therefore the equation (\ref{s18}) reduces to%

\begin{equation}
\frac{d\left(  \cos\tau_{0}-k_{2}\sin\tau_{0}\right)  }{ds}=0. \label{s19}%
\end{equation}

\noindent Then the nonzero term $\cos\tau_{0}-k_{2}\sin\tau_{0}$ in
(\ref{s19}), must be \ a constant. So we can write%

\begin{equation}
\cos\tau_{0}-k_{2}\sin\tau_{0}=\delta\neq0 \label{w1}%
\end{equation}

\noindent where $\delta$ is a constant.

\noindent If we wr\i te%

\begin{equation}
\frac{\lambda}{\delta}=\ell_{0}\neq0 \label{s21}%
\end{equation}

\noindent where $\ell_{0}$ is a constant. By using (\ref{s16}), (\ref{s17})
and (\ref{s21}), we get%

\begin{equation}
\overline{L}\left(  \overline{s}\right)  =\frac{1}{\ell_{0}}L\left(  s\right)
, \label{s22}%
\end{equation}

\begin{equation}
\overline{N}\left(  \overline{s}\right)  =\ell_{0}N\left(  s\right)  .
\label{s23}%
\end{equation}

\noindent Differentiating (\ref{s23}) with respect to $s$, we get%

\begin{equation}
\left(  \overline{W_{1}}\left(  \overline{s}\right)  -\overline{k_{2}}\left(
\overline{s}\right)  \overline{W_{2}}\left(  \overline{s}\right)  \right)
\lambda=\ell_{0}\left(  W_{1}\left(  s\right)  -k_{2}\left(  s\right)
W_{2}\left(  s\right)  \right)  . \label{s25}%
\end{equation}

\noindent From (\ref{s25}), we have%

\begin{equation}
(1+\left(  \overline{k_{2}}\left(  \overline{s}\right)  \right)  ^{2}%
)\lambda^{2}=(\ell_{0})^{2}\left(  1+(k_{2}\left(  s\right)  )^{2}\right)  .
\label{s26}%
\end{equation}

\noindent Using (\ref{w1}) and (\ref{s21}) into (\ref{s26}), we obtain%

\[
\left(  \overline{k_{2}}\left(  \overline{s}\right)  \right)  ^{2}=\left[
\frac{\sin\tau_{0}+k_{2}\left(  s\right)  \cos\tau_{0}}{\cos\tau_{0}%
-k_{2}\left(  s\right)  \sin\tau_{0}}\right]  ^{2}.
\]

\noindent Let us take%

\begin{equation}
\overline{k_{2}}\left(  \overline{s}\right)  =\frac{\sin\tau_{0}+k_{2}\left(
s\right)  \cos\tau_{0}}{\cos\tau_{0}-k_{2}\left(  s\right)  \sin\tau_{0}}.
\label{s27}%
\end{equation}

\noindent If we use (\ref{s13}),\ (\ref{w1}), (\ref{s21}) and (\ref{s27}) into
(\ref{s25}), we get%

\begin{equation}
\overline{W_{2}}\left(  \overline{s}\right)  =-\sin\tau_{0}W_{1}\left(
s\right)  +\cos\tau_{0}W_{2}\left(  s\right)  . \label{s28}%
\end{equation}

And it is trivial that the Cartan (1,2)-normal plane at each point $c\left(
s\right)  $ of $C$ coincides with the Cartan (1,2)-normal plane at
corresponding point $\overline{c}\left(  \overline{s}\right)  $ of
$\overline{C}$. Therefore $C$ is a (1,2)-Bertrand curve in $R_{1}^{4}$.
\end{proof}

\section{An example of 2-degenerate (1,2)-Bertrand curve in $R_{1}^{4}$}

\noindent\ \ \ \ Let $C$ be a curve in $R_{1}^{4}$ defined by%

\[
c\left(  s\right)  =\left(  \frac{s^{2}}{2},s,\frac{s^{2}}{2},1\right)  .
\]

\noindent Then we get the Cartan frame and the Cartan curvatures as follows:%

\begin{align*}
W_{1}\left(  s\right)   &  =\left(  s,1,s,1\right)  ,\\
L\left(  s\right)   &  =\left(  1,0,1,0\right)  ,\\
W_{2}\left(  s\right)   &  =\left(  \sqrt{3},0,\sqrt{3},1\right)  ,\\
N\left(  s\right)   &  =\left(  \frac{s^{2}}{2}+2,s,\frac{s^{2}}{2}+1,\sqrt
{3}\right)  ,\\
k_{1}\left(  s\right)   &  =k_{2}\left(  s\right)  =0.
\end{align*}

\noindent Now we choose the polynomial functions $\alpha$ and $\beta$ as follows:%

\begin{align*}
\alpha\left(  s\right)   &  =0,\\
\beta &  :R-\left\{  0\right\}  \rightarrow R;\text{ \ }\beta\left(  s\right)
=\sqrt{3}s.
\end{align*}

\noindent Its Bertrand mate is given by%

\[
\overline{c}\left(  \overline{s}\right)  =\left(  \frac{(\overline{s})^{2}}%
{8}+\frac{3\overline{s}}{2},\frac{\overline{s}}{2},\frac{(\overline{s})^{2}%
}{8}+\frac{3\overline{s}}{2},1+\frac{\sqrt{3}}{2}\overline{s}\right)
\]

\noindent where $\overline{s}$ is the pseudo-arc parameter of $\overline{C}$,
and a regular map $\varphi:s\rightarrow\overline{s}$ is given by%

\[
\overline{s}=\varphi\left(  s\right)  =2s.
\]

{\noindent Mehmet G\"{o}\c{c}men }

{\noindent Department of Mathematics, Faculty of Arts and Sciences,
\.{I}n\"{o}n\"{u} University }

{\noindent44280 Malatya, Turkey.}

\noindent Email: mgocmen1903@gmail.com

{\noindent Sad\i k Kele\c{s} }

{\noindent Department of Mathematics, Faculty of Arts and Sciences,
\.{I}n\"{o}n\"{u} University }

{\noindent44280 Malatya, Turkey }

{\noindent Email: skeles@inonu.edu.tr }

\end{document}